\documentclass[12pt,reqno]{amsart}
\usepackage[a4paper,bindingoffset=0.1in,left=1.2in,right=1.3in,top=1.4in,bottom=1.5in,footskip=.25in]{geometry}
\usepackage{indentfirst,amssymb,amsfonts,amsmath,amsthm}    
\usepackage{amsmath,amssymb}
\usepackage{newtxtext,newtxmath}
\usepackage{setspace}
\usepackage{times}
\usepackage[utf8]{inputenc}
\usepackage{stackengine}
\usepackage[T1]{fontenc}
\usepackage{verbatim}
\usepackage{hyperref}
\hypersetup{colorlinks=true, linkcolor=blue,citecolor=blue, urlcolor=blue}
\usepackage{mathtools}
\numberwithin{equation}{section}

\newtheorem{thm}{Theorem}[section]

\newtheorem{prop}{Proposition}[section]

\newtheorem{coro}{Corollary}[section]

\DeclareMathOperator{\tr}{Tr}
\DeclareMathOperator{\mtr}{tr}
\DeclareMathOperator{\gal}{Gal}

\DeclareMathOperator{\End}{End}

\makeatletter
\g@addto@macro{\endabstract}{\@setabstract}
\newcommand{\authorfootnotes}{\renewcommand\thefootnote{\@fnsymbol\c@footnote}}%
\makeatother

\begin{document}
\setcounter{page}{1} 

\baselineskip .65cm 
\pagenumbering{arabic}
\title{    Linear Algebra and Galois Theory}
\author [Ashish Gupta and Sugata Mandal]{Ashish Gupta and Sugata Mandal}

\address{Ashish Gupta, Department of Mathematics\\
Ramakrishna Mission Vivekananda Educational and Research Institute (Belur Campus) \\
Howrah, WB 711202\\
India}
\email{a0gupt@gmail.com \thanks{}}
\address{Sugata Mandal, Department of Mathematics\\
Ramakrishna Mission Vivekananda Educational and Research Institute (Belur Campus) \\
Howrah, WB 711202\\
India}
\email{gmandal1961@gmail.com\thanks{}}

\maketitle
\begin{abstract}
   In \cite{GQ2008} R. Gow and R. Quinlan have cast a new look on the endomorphism algebra of a $K$-vector space $V$ of dimension $n$ assuming that $K$ has a Galois extension $L$ of degree $n$. In this approach the $K$-space $L$ may serve as a model  for $V$ and Galois-theoretic ideas and results may be applied to elucidate the structure of endomorphisms and other important objects of linear algebra. In particular, this leads to the clarification of the structure of a rank-one endomorphism, trace of an endomorphism, criteria for linear indepedence etc. We present an exposition of these results using the language of tensor algebra wherever possible to provide shorter and more conceptual proofs.
   
\noindent\textbf{2020 Math. Subj.
Class.}: 12F10; 15A03; 15A04
\end{abstract}
\section{Introduction}

%hey gave a new charaterization of rank oneo endomorphisms and also gave a new formula for the trace $\tr(\tau)$ of an endomorphism $\tau$ in terms of the Galois-theoretic trace $\tr$ as defined in \cite{GQ2008}.  

Let $K$ be a field (with arbitrary characteristic) admitting a Galois extension $L$ of degree $n$ and let $V$ be an $n$-dimensional $K$-space. Then $L$ may serve as a model for $V$. This allows for a new look on the fundamental notions and concepts of linear algebra in the light of Galois theory. Such was the goal of the study undertaken in \cite{GQ2008} where there was a special focus on endomorphisms and their properties as well certain determinants that can be formed using the elements of the Galois group. Some of the results below require $L$ to be cyclic.

We begin by noting as in \cite{GQ2008} that any linear functional $f \in L^\ast$ may be written in the form $ f = f_b$ for suitable $ b \in L$, where $f_b(x) = \tr (bx)$ and $\tr$ is the trace map of $L/K$ defined as 
\[ \tr (x)= \displaystyle \sum_{\sigma \in \gal(L/K)} \sigma(x). \] 
Indeed
$L^\ast = \hom_K(L,K)$ is an  $L$-space with the action $f \mapsto f.b$ where   
\[ f.b(x) = f(bx). \]
Moreover, it is clear that $\dim_L(L^\ast) = 1$ and as $\tr$ is in general a nonzero functional in $L^\ast$, therefore $L^\ast = \tr .L$.  Let $\theta : L \rightarrow L^\ast$ be the map defined by $b \rightarrow \tr.b \ (b \in L)$. As the $\tr$ bilinear form of a  Galois extension is non-degenerate it follows that $\theta$ is an isomorphism of $K$-spaces. 
In particular, $\{\tr .b_1, \cdots, \tr .b_n \}$ is a K-basis of $L^\ast$ whenever $b_1, \cdots, b_n$ is a K-basis of $L$.

Let $V$ and $W$ be $K$-spaces. Recall that an element $z \in V \otimes_K W$ is called \emph{decomposable} if it decomposes as \[ z = x \otimes y, \quad x\in V,y \in W.\]
In the following $\otimes$ always means  $\otimes_K$.  Recall that the space  $$ \displaystyle\mathcal{T}_{p}^{q}(V) =( \otimes^p V) \otimes ( \otimes^q V^{\ast})$$ is called the space of tensors of type $(p,q)$ on $V$. Given $a \in V$ and $f \in V^{\ast}$, the \emph{contraction} of the $(1,1)$-tensor $(a \otimes f) \in V \otimes V^{\ast}$ means $f(a)$.
\subsection{Rank-one endomorphisms and a decomposition of $\End_K(L)$} \label{rank-1-endo}

%If $ x = \sum_{i=1}^{m} x_i e_i$, $ y = \sum y_{j=1}^{m} f_j$ where $\{e_i: 1\leq i \leq m\}$ and $ \{f_j: 1 \leq j \leq n\}$ are the bases of vector spaces $V$ and $W$, then $z$ can be written as $z = \sum_{i,j} z_{ij} e_i \otimes f_j$, with $z_{ij} = x_iy_j$. Let us assume that $ W = V^{\times}$. 

 As is well-known $ L \otimes L^{\ast}\cong\End_K(L)  $ as $K$-spaces via the map $(a \otimes f) \rightarrow \psi_{a,f}$, where $ \psi_{a,f}(x)= f(x) a \ (x \in L)$.  Moreover the nonzero decomposable elements correspond precisely to the rank-one endomorphisms. It follows that the latter may be written as $a \otimes \tr.b$ for suitable  $a, b \in L$. Thus $( a \otimes \tr.b ) (x) = \tr (bx)a$ for all $x \in L$. Since the contraction of a $(1,1)$-tensor is its trace (as a linear operator) it follows that $\mtr( a \otimes \tr.b)= \tr(ab)$. %Since any operator $\tau \in \End_K(L)$ can be expressed as a sum of $(1,1)$-tensors, it follows that we can compute the trace of $\tau$.
 If $\{b_1,b_2,\ldots,b_n\}$ is a basis of  the $K$-space $L$ then as mentioned previously    $\{ \tr.b_1, \tr.b_2,\ldots, \tr.b_n\}$ forms a $K$-basis of $ L^{\ast}$ thus by defining property of the tensor product  an endomorphism $\tau \in \End_{K}(L)$ has a  expression: $$\tau= a_1 \otimes \tr.b_1 + a_2 \otimes \tr.b_2 + \ldots + a_n \otimes \tr.b_n,$$ where $ a_1,\ldots a_n \in L$ are uniquely determinable. It follows that $$\mtr(\tau)= \sum_{i=1}^{k}\tr(a_ib_i).$$
% Since any operator $\tau \in \End_K L$ is  of the form \[\tau = \sum_{i=1}^{k}(a_i \otimes \tr.b_i),\] it follows that $$\mtr(\tau)= \sum_{i=1}^{k}\tr(a_ib_i).$   % Now suppose that $\{b_1,\ldots,b_n\}$ is a $K$-basis of $ L^{\ast}$. Then  $$ \tr.b_1, \tr.b_2,\ldots, \tr.b_n$$ is a $K$-basis of $L^{\ast}$. By the defining property of a tensor $\tau$ has a unique expression: $$ a_1 \otimes \tr.b_1 + a_2 \otimes \tr.b_2 + \ldots + a_n \otimes \tr.b_n,$$ where $ a_1,\ldots a_n \in L$ can be determined uniquely.
\subsection{Dual bases}  
%Also the rank $1$, $K$-endomorphism is characterised by decomposible element of $ L \otimes L^*$, equivalently the element of the form $\psi_{c,\phi_b}$. Thus any element of $\End_{K}L$ can be expressed uniquely $ \displaystyle \sum_{i=1}^{n} \psi_{e_i,\phi_{b_i}}$, where $\{e_i: 1 
%\leq i \leq  n\}$ be a basis of the $K$-space $L$. 
Let $\{e_1,\cdots,e_n\}$ be a basis of the $ K$-space $ L$ and $ \{\epsilon_1,\cdots,\epsilon_n\}$ be the corresponding dual basis of $ L^{\ast}$. Then $ \epsilon _i$ may be expressed in the form $\tr. c_i$ where $ c_i \in L$. Let $c_i = \displaystyle \sum_{j=1}^{n} u_{ij} e_j$. Therefore $\forall i$,  \begin{align*}
   \delta_{ik}= ( \tr c_i)(e_k)= \tr (\sum_{j=1}^{n} (u_{ij} e_j)e_k) = \sum_{j=1}^{n} u_{ij} \tr (e_je_k).   
\end{align*}
Hence $U=(u_{ij}) = A^{-1}$, where the matrix $ A =( \tr(e_j e_k))$.

For example, for the cyclotomic extension $ \mathbb{Q} (\eta_{p})$ over $ \mathbb{Q}$ where $ \eta_p$ is a primitive $ p$-th root of unity we can take the basis
$ e_i = \eta_p^{i-1}$ with $1\leq i \leq n$. Then by \emph{\cite[Example 9.1.121]{vb}} we have
$$
\tr(e_i e_k)=
\begin{cases}
p-1, \quad  \text{if} ~i+k -2 \equiv 0~ ( 
 \text{mod}~p ), \\
-1,  \quad \text{otherwise}.
\end{cases}
$$
 By direct computation for $p  \geq 5 $ the matrix $U$ is as follows: 
 
 $$ U= { \begin{pmatrix}
    X & Y \\
    Y^{t} & Z
\end{pmatrix}}_{(p-1) \times (p-1)},$$
where $$ X= {\begin{pmatrix}
    p^{-1} & -p^{-1}\\
    -p^{-1} & -2p^{-1}
\end{pmatrix}}_{2\times 2}, $$
$$ Y = {\begin{pmatrix}
    0 & 0 & \ldots 0\\
    -p^{-1} & -p^{-1} & -p^{-1}
\end{pmatrix}}_{2 \times (p-3)} $$
and 
 \[ Z = { \begin{pmatrix}
    0 & 0 & \ldots &  0 & p^{-1}\\
    0 & 0 & \ldots  &p^{-1} & 0\\
    \vdots & \vdots & \vdots & \vdots & \vdots \\
     p^{-1} &0 & \ldots &  0 & 0.
\end{pmatrix}}_{(p-3) \times (p-3)}. \]

\subsection{Determinants associated to hyperplanes}

The following is a well-known fact.
\begin{prop}\label{ condition for invertible}
    Let $L$ be a Galois extension of $K$ of degree $n$ with Galois group $G=\{\sigma_1,\ldots,\sigma_n\}$.  Then $ \{x_1,x_2,\ldots, x_n\}$ is a $ K$-basis for $ L$ if and only if $ \sigma_j(x_i)$ is invertible.
\end{prop}
\begin{proof}
    Easy consequence of Dedekind's independence theorem.
\end{proof}
The following  theorem was shown in \cite{GQ2008}.
\begin{thm}\emph{(\cite[Theorem 5]{GQ2008})} \label{indepndent det}
    Let $L$ be a Galois extension of $K$ of degree $n$ with Galois group $G=\{\sigma_1,\ldots,\sigma_n\}$. Let $ x_1, \ldots,x_n$ be a $K$-basis of $L$. Let $B$ be the $ n \times n$ matrix  whose $ i,j$ entry is $\sigma_j(x_i)$, $1 \leq i,j \leq n$.  Then
    \begin{itemize}
        \item[(i)] if $K$ has characteristic $2$ or if the Sylow $2$-subgroup of $G$ is not cyclic, then $ \text{det}~ B \in K$. 
        \item[(ii)] if $K$ has characteristic  diferent from $2$ and $G$ has a cyclic Sylow $2$-subgroup, then $ \text{det}~B = \mu \alpha$, where $\mu \in K$, $\alpha^2 \in K$ and $ K(\alpha)$ is the unique quadratic extension of $ K$ contained in $L$.
    \end{itemize}
\end{thm}
 An analog of Theorem \ref{indepndent det} using an $ (n-1) \times (n-1)$ matrix formed from a basis of the trace zero hyperplane $ H_0= \ker \tr$ was also shown in \cite{GQ2008}.  Let $ H_0$ be the trace zero hyperplane, that is, 
 $ H_0 = \ker \tr $. Let $ b_2,\ldots, b_n$ be a $K$-basis of $H_0$. We can extend $ \{b_2,\ldots, b_n \}$ to a $K$-basis $\{b_1,b_2\ldots, b_n\}$ of $L$ for some $ b_1 \notin H_0$. Let $E(x)$ be the $n\times n$ matrix defined by 
 \begin{equation*}
   E(x):= \begin{pmatrix}
        \sigma_1(x)& \sigma_2(x)&\ldots& \sigma_n(x)\\
        \sigma_1(b_2)& \sigma_2(b_2)&\ldots& \sigma_n(b_2)\\
        \vdots& \vdots&\ldots& \vdots\\
        \sigma_1(b_n)& \sigma_2(b_n)&\ldots& \sigma_n(b_n)       
    \end{pmatrix} \qquad \text{for all $x \in L$}.
\end{equation*}
We now define an element $\theta \in \End_{K}(L)$ by setting $$\theta(x)= \det E(x).$$  We note that  $ \theta(b_1) = \det E(b_1) $ is non-zero by Proposition \ref{ condition for invertible}. Thus $\theta $ is a non-zero endomorphism and moreover it is easily checked that $ \theta $ vanishes on the  hyperplane $ H_0$. Consequently $ \theta $ is a rank-$ 1$ endomorphism and  by Section \ref{rank-1-endo}, $ \theta$ is of the form $c\otimes \tr.b$ for some $ c \in L $ and since $\ker(\theta) = H_0$ so we can take $b=1
$.  We now determine the value of $c$. Let $E_1$ be the $ (n-1) \times (n-1)$ matrix obtained by omitting the first row and the first column of $ E$. Thus \begin{equation*}
   E_1= \begin{pmatrix}
        \sigma_2(b_2)& \ldots& \sigma_n(b_2)\\
        \vdots& \ldots& \vdots\\
        \sigma_2(b_n)& \ldots& \sigma_n(b_n)       
    \end{pmatrix}
\end{equation*}
\begin{thm}\emph{(\cite[Corollary 2]{GQ2008})}
     With the above notation  we have $c = \det E_1$.
\end{thm}
\begin{proof}
    Let $E'(x)$ be the $n \times n$ matrix obtained by summing all the columns of $E(x)$ and placing the result as the first column (the remaining coloumns being unaltered). 
    Thus  \begin{equation*}
   E'(x):= \begin{pmatrix}
       \displaystyle \sum_{\sigma_i \in G}\sigma_i(x)& \sigma_2(x)&\ldots& \sigma_n(x)\\
       \displaystyle \sum_{\sigma_i \in G}\sigma_i(b_2)& \sigma_2(b_2)&\ldots& \sigma_n(b_2)\\
        \vdots& \vdots&\ldots& \vdots\\
      \displaystyle \sum_{\sigma_i \in G}\sigma_i(b_n)& \sigma_2(b_n)&\ldots& \sigma_n(b_n)       
    \end{pmatrix} = \begin{pmatrix}
       \displaystyle \tr(x)& \sigma_2(x)&\ldots& \sigma_n(x)\\
       \displaystyle 0& \sigma_2(b_2)&\ldots& \sigma_n(b_2)\\
        \vdots& \vdots&\ldots& \vdots\\
      \displaystyle 0& \sigma_2(b_n)&\ldots& \sigma_n(b_n)       
    \end{pmatrix}
\end{equation*} 
    Then \[\theta(x)=\det E(x) = \det E'(x) = \det E_1\tr(x),\]  whence $ \theta = \det E_1 \otimes \tr $ and consequently $ c = \det E_1$.
\end{proof}
Using the preceding result and some basic group theory the following is easy to deduce.
\begin{coro}\emph{(\cite[Theorem 6]{GQ2008})} \label{indepndent det hyperplane}
    Let $L$ be a Galois extension of $K$ of degree $n$ with Galois group $G=\{\sigma_1,\ldots,\sigma_n\}$. Let $ H_0$ be the trace-zero $K$-hyperplane in $L$. Let $ b_2, \ldots,b_n$ be a $K$-basis of $H_0$. Then the $ (n-1) \times (n-1)$ matrix $E_1$ whose $ i,j$ entry is $\sigma_j(b_i)$, $2 \leq i,j \leq n$, is invertible.  Moreover,
    \begin{itemize}
        \item[(i)] if $K$ has characteristic $2$ or if the Sylow $2$-subgroup of $G$ is not cyclic, then $ \text{det}~ E_1 \in K$. 
        \item[(ii)] if $K$ has characteristic  diferent from $2$ and $G$ has a cyclic Sylow $2$-subgroup, then $ \text{det}~E_1 = \mu \alpha$, where $\mu \in K$, $\alpha^2 \in K$ and $ K(\alpha)$ is the unique quadratic extension of $ K$ contained in $L$.
    \end{itemize}
\end{coro}
    Next we consider any $K$-hyperplane $H$ in $L$. Then $H$ arises as the kernel of a non-zero element in $L^{\ast}$. As $L^{\ast}= \tr.L$, therefore, $H=\ker \tr.a$ for some non-zero $a\in L$. Then $H = a^{-1}H_0$ and $c_2,\ldots,c_n$ is a basis of $H$ where $ c_i= a^{-1} b_i$. We can extend $ \{c_2,\ldots, c_n \}$ to a $K$-basis $\{c_1,c_2\ldots, c_n\}$ of $L$.  We  define a matrix $ D(x)$ by replacing $ b_i$ by $c_i$  in the matrix $E(x)$ above and $\theta_H \in \End_{K}(L)$ by setting $\theta_H(x)= \det D(x)$.  We note that  $ \theta_H(c_1) = \det D(c_1) \neq 0$ by Proposition \ref{ condition for invertible}. Thus $\theta_H \neq 0$ and moreover $ \theta_H $ vanishes on the  hyperplane $ H$. Consequently $ \theta_H $ is a rank-$ 1$ endomorphism  $d\otimes \tr.a$ for some nonzero  $ d \in L $.
    Recall that if $F$ is an intermediate subfield and $a \in L$ then the $L/F$-norm  
 $N_{L/F}(a)$ of $a $ is defined as $N_{L/F}(a) = \displaystyle\prod_{\theta \in \gal(L/F)} \theta(a)$. Clearly $N_{L/F}(a) \in F$.
 \begin{thm}
     With the notation as previously introduced, we have $d = N_{L/K}(a^{-1}) \det E_1$.
\end{thm}
\begin{proof}
   % Let $D'(x)$ be the $n \times n$ matrix obtained by multiplying each column by $\sigma_i(a)$ and summing all the columns of $D(x)$ and placing the result as the first column of $D'(x)$. Then 
    \begin{align*}
    \theta_H
    (x) &=\det D(x)\\
    &= \det \begin{pmatrix}
        \sigma_1(x)& \sigma_2(x)&\ldots& \sigma_n(x)\\
        \sigma_1(c_2)& \sigma_2(c_2)&\ldots& \sigma_n(c_2)\\
        \vdots& \vdots&\ldots& \vdots\\
        \sigma_1(c_n)& \sigma_2(c_n)&\ldots& \sigma_n(c_n)       
    \end{pmatrix}\\
    &= \det \begin{pmatrix}
       \sigma_1(a^{-1}) \sigma_1(ax)& \sigma_2(a^{-1}) \sigma_2(ax)&\ldots& \sigma_n(a^{-1}) \sigma_n(ax)\\
        \sigma_1(a^{-1}b_2)& \sigma_2(a^{-1}b_2)&\ldots& \sigma_n(a^{-1}b_2)\\
        \vdots& \vdots&\ldots& \vdots\\
        \sigma_1(a^{-1}b_n)& \sigma_2(a^{-1}b_n)&\ldots& \sigma_n(a^{-1}b_n)       
    \end{pmatrix}\\
    &= N_{L/K}(a^{-1})\det D'(x)\\
    &= N_{L/K}(a^{-1}) \det E_1\tr(ax)\\
    &=  N_{L/K}(a^{-1}) \det E_1 \otimes \tr(ax),
    \end{align*}
   where $D'(x)$ be the $n \times n$ matrix obtained by summing all the columns of $D(x)$ and placing the result as the first column (the remaining coloumns being unaltered). 
    Consequently $ d =  N_{L/K}(a^{-1})\det E_1$.
\end{proof}

\subsection{Criteria for Linear Independence in a Cyclic Extension}
\begin{thm}
    Let $L $ be a cyclic Galois extension of $K$ of degree $n$ with  $\gal(L/K) = \langle \sigma \rangle$. Let $k$ be an integer satisfying $1\leq k \leq n$ and let $b_1,\ldots,b_k$ be elements of $L$. Then these elements are linearly dependent over $K$ if and only if \[\det S = 0,\]
    where 
    \begin{equation*}
      S=\left(\sigma^{i-1}(b_j)\right) = \begin{pmatrix}
            b_1 & b_2 & \ldots & b_n\\
            \sigma(b_1) & \sigma(b_2) & \ldots & \sigma(b_k)\\
           \vdots& \vdots&\ldots& \vdots\\
           \sigma^{k-1}(b_1) & \sigma^{k-1}(b_2) & \ldots & \sigma^{k-1}(b_k)
        \end{pmatrix}.
    \end{equation*}
\end{thm}
\begin{proof}
  If $ b_1,\ldots, b_k$  are linearly dependent then it is easily seen   that $\det S = 0$.  Conversely, let $b_1,\ldots,b_k$ be linearly independent over $K$. We will show that $ \det S \neq 0$ by induction on $k$.  It is clear that the assertion holds for $ k=1$. 

  Note that it is sufficient to show that $\det S' \neq 0$, where
  \begin{equation*}
       S'= \begin{pmatrix}
        \sigma^{k-1}(b_1) & \sigma^{k-1}(b_2) & \ldots & \sigma^{k-1}(b_k)\\
        \sigma^{k-2}(b_1) & \sigma^{k-2}(b_2) & \ldots & \sigma^{k-2}(b_k)\\ 
           \vdots& \vdots&\ldots& \vdots\\
           b_1 & b_2 & \ldots & b_k
        \end{pmatrix}
    \end{equation*}
    is obtained from $S$ by row-interchange operations.  
   For $i= k-1,\ldots,1$ we successively substract $\gamma_i \times \text{row $i$}$ from row $i+1$ where $$\gamma_i = \sigma^{k-i-1}(b_1) (\sigma^{k-i}(b_1))^{-1}= \sigma^{k-i-1}(b_1(\sigma(b_1))^{-1}).$$ The resulting matrix is
  \begin{equation*}
       U= \begin{pmatrix}
        \sigma^{k-1}(b_1) & \sigma^{k-1}(b_2) & \ldots & \sigma^{k-1}(b_k)\\
        0 & \sigma^{k-2}(x_2) & \ldots & \sigma^{k-2}(x_k)\\ 
           \vdots& \vdots&\ldots& \vdots\\
           0 & x_2 & \ldots & x_k
        \end{pmatrix}.
    \end{equation*}
     where 
     \begin{equation} \label{new_var}
     x_j= b_j - \frac{b_1}{\sigma(b_1)} \sigma(b_j), \qquad   2 \leq j \leq k. \end{equation}
     Then \begin{equation*}
         \det U = \sigma^{k-1}(b_1) \det V,
     \end{equation*}  
  where $V$ be the $(k-1) \times (k-1)$ matrix obtained by deleting the first row and column of $U$.  We also note that the $i,j$-entry of $V $ is $$\sigma^{k-1+i}(x_j),~~~~~ 1\leq i \leq k-1, ~ 1 \leq j \leq k, $$ where $$ x_j= b_j - \frac{b_1}{\sigma(b_1)} \sigma(b_j).$$ 

  Clearly $\det(S') \ne 0$ if and only if $\det(V) \ne 0$.
  As $V$ has the same form as $S'$ by the induction hypothesis $\det(V) = 0$ implies that $x_2, \cdots, x_k$ are linearly dependent, that is, $$ a_2 x_2  + \ldots + a_k x_k = 0,$$ for suitable scalars 
  $a_2,\ldots,a_k$ in $ K$, not all zero. Thus by \eqref{new_var}
  \begin{align*}
      a_2 b_2 + \ldots a_k b_k &= \frac{b_1}{\sigma(b_1)} \sigma( a_2 b_2 + \ldots a_k b_k)
        \end{align*}
        whence
        \[
      \sigma\left(\frac {a_2 b_2 + \ldots a_k b_k}{ b_1}\right) = \frac {a_2 b_2 + \ldots a_k b_k}{ b_1} \] 

  It follows that $\frac {a_2 b_2 + \ldots a_k b_k}{ b_1}  \in K$, thus contradicting the linear independence of $b_1, \cdots, b_k$.  The theorem now follows.
\end{proof}
\section*{Declaration of competing interest}
The authors declare that they have no known competing financial interests or personal
relationships that could have appeared to influence the work reported in this paper.
\section*{Acknowledgements}
The  second author thanks the National Board of Higher Mathematics (NBHM) for financial support.

\end{document}